\documentclass[11pt]{article}

\usepackage{amssymb}

\usepackage{t1enc}
\usepackage[latin1]{inputenc}
\usepackage{graphicx}
\usepackage{amsthm}
\usepackage[T1]{fontenc}
\usepackage{amsmath}
\usepackage{amssymb}
\usepackage{stmaryrd}
\usepackage{enumerate}
 \usepackage[all]{xypic}
\usepackage[mathscr]{eucal}
\usepackage[german,UKenglish]{babel}

\newtheorem{thm}{Theorem}
\newtheorem{lem}[thm]{Lemma}
\newtheorem{rmk}[thm]{Remark}
\newtheorem{defi}[thm]{Definition}
\newtheorem{cor}[thm]{Corollary}

\DeclareMathOperator{\spann}{span}\DeclareMathOperator{\tr}{tri}
\DeclareMathOperator{\pr}{pr}
\DeclareMathOperator{\sgn}{sgn}

\newcommand{\C}{\mathbb{C}}
\newcommand{\dirac}{\mathcal{D}}

\newcommand{\im}{\,\mathtt{Im}\,}

\newcommand{\ci}{\mathtt{i}}
\newcommand{\Spinc}{\text{Spin}^{\C}}
\newcommand{\Peins}{{\mathbb{P}^1}}

\newcommand{\HQ}{\mathbb{H}}
\newcommand{\R}{{\mathbb{R}}}
\newcommand{\Z}{{\mathbb{Z}}}

\newcommand{\eps}{\varepsilon}

\newcommand{\tangent}{\text{T}}

\newcommand{\sep}{;\,}

\title{$\Spinc$ Dirac operators over the flat 3-torus}

\author{J. Fabian Meier \\ Universität Bonn}

\begin{document}

\maketitle




\begin{abstract}
We determine spectrum and eigenspaces of some families of $\Spinc$ Dirac operators over the flat 3-torus. Our method relies on projections onto appropriate 2-tori on which we use complex geometry.

Furthermore we investigate those families by means of spectral sections (in the sense of Melrose/Piazza). Our aim is to give a hands-on approach to this concept. First we calculate the relevant indices with the help of spectral flows. Then we define the concept of a \emph{system of infinitesimal spectral sections} which allows us to classify spectral sections for small parameters $R$ up to equivalence in $K$-theory. We undertake these classifications for the families of operators mentioned above.

Our aim is therefore twofold: On the one hand we want to understand the behaviour of $\Spinc$ Dirac operators over a 3-torus, especially for situations which are induced from a 4-manifold with boundary $T^3$. This has prospective applications in generalised Seiberg-Witten theory. On the other hand we want to make the term ``spectral section'', for which one normally only knows existence results, more concrete by giving a detailed description in a special situation. 
\end{abstract}

~\\[0.1ex] {\bf Keywords:}\\[0.5ex]
Spinc Dirac operator\sep 3-torus\sep spectral section \\[1ex]
{\bf Subject classification:} \\[0.5ex]
MSC[2010] 47A10\sep 58C40\sep 58J30





\section{Introduction}
\label{sec:intro}

In the study of smooth 4-manifolds, especially in the context of (generalised) Seiberg-Witten theory, it would be nice to understand $\Spinc$ Dirac operators which are induced on the boundary of a compact 4-manifold.

Manifolds with boundary $T^3$ where already studied in this context by \cite{swtorus}. But for generalized Seiberg-Witten theories, also families of operators in non-trivial $\Spinc$ structures become important. Therefore, we undertake a detailed study for some of these families. We now describe the object of investigation:

For every $\Spinc$ structure on $T^3= \R^3/\Z^3$ we analyse the family of Dirac operators given by connections $\nabla^K + \ci\alpha$; here $\nabla^K$ is a fixed background connection (to be constructed below) for an appropriate line bundle $K$ and $\alpha$ comes from the parameter space of closed one-forms.

Our first aim is to determine the spectrum and an orthogonal eigenbasis for these operators. Our strategy is as follows:
\begin{enumerate}
\item We write the 3-torus as $S^1$ bundle over a 2-torus (determined by the $\Spinc$ structure).
\item We equip the 2-torus with a complex structure and choose appropriate holomorphic line bundles.
\item We use complex geometry and methods from \cite{tejeroprieto}.
\item We combine the calculated terms with exponential functions to get the desired result.
\end{enumerate}

The calculations above will help us to access our second aim: The construction of spectral sections.

For a lattice $\ell \subset H^1(T^3;\Z) \subset H^1(T^3;\R)$ look at the family of operators parametrised by $B=\big(\ell \otimes \R\big)/\ell$. Since we know the concrete spectrum we can calculate all spectral flows in this torus which gives us direct access to the index in $K^1(B)$. By \cite[section 2]{melrosepiazza} the vanishing of this index corresponds to the existence of spectral sections.

For small parameters $R$ we give a classification of all spectral sections up to equivalence in $K$-theory.

\begin{rmk}
  If $\iota: T^3 \hookrightarrow M$ is the boundary of a $\Spinc$ 4-manifold $M$ and $\ell$ is chosen to be a subset of $\iota^*\big(H^1(M;\Z)\big)$, then one can show that our family of operators is a boundary family in the sense of \cite{melrosepiazza}; this guarantees the existence of spectral sections in this case but does not lead to concrete constructions of them.
\end{rmk}

\section{Definitions}
\label{sec:defini}

We take $T^3:= \R^3/\Z^3$ to be the flat 3-torus. We identify the first and second cohomology groups with each other by the Hodge star operation. Both of them will be identified with $\Z^3$ or $\R^3$ through the standard (positively oriented) basis $dx_1, dx_2, dx_3$ of $\tangent \R^3$.

The trivial Spin structure induces a $\Spinc$ structure with associated bundle $\underline \HQ = T^3 \times \HQ$. Here $\HQ= \spann\{e_0, e_1, e_2, e_3\}$ denotes the space of quaternions. It is considered as a complex vector space by left multiplication with $\ci = e_1$ and as a left-quaternionic vector space by inverse right multiplication.

Now the $\Spinc$ structures can be canonically identified with elements $\hat k \in H^2\big(T^3;\Z\big)$ (for a general explanation of $\Spinc$ structures and their associated bundles see e.g. \cite{morgan}). For every such element we choose a Hermitian line bundle $K$ with $c_1(K)= \hat k$ and a unitary background connection $\nabla^K$; possible choices and constructions will be detailed in the subsequent sections. Then the $\Spinc$ structure $\hat k$ has the associated bundle $\HQ \otimes K$.

For each $K$ and closed one-form $\alpha$ we get a $\Spinc$ Dirac operator
\begin{align*}
  \dirac^K_\alpha : \Gamma\big(\underline \HQ \otimes K\big) \to \big(\underline \HQ \otimes K\big)
\end{align*}
for the connection $\nabla^K+\ci \alpha$.

These operators will be analysed in the subsequent sections.

\section{Spectrum and Eigenbasis}
\label{sec:speceig}

We distinguish two main cases.

\subsection{Nontrivial $\Spinc$ structure}
\label{sec:specnon}

We write $\hat k = h \cdot k$ with $k\in \Z^3$ and maximal $h\in \Z^+$. Let $W$ be the plane in $\R^3$ orthogonal to $k$ and $\pi_k$ the orthogonal projection. By taking quotients we get a map $\pi_{\overline k} : T^3 \to T_\Lambda := W/\Lambda$ with $\Lambda= \pi_{k}(\Z^3)$.

Let $w_1, w_2$ be the basis of a fundamental parallelogram in $\Lambda$. We take $c^i \in [0,1)$, $i=1,2$, with $w_i - c^i\cdot k \in \Z^3$.

\begin{lem}
   The map $\pi_{\overline k} : T^3 \to T_\Lambda$ determines a trivial $\R/\Z$-bundle with trivialisation:
   \begin{align}\label{trivmap}
     \begin{array}{rcl}
      T^3 & \stackrel{\pi_{\overline k} \times \tr}{\longrightarrow} & T_\Lambda\times \R/\Z \\ 
     \bigg[\chi_1w_1 + \chi_2w_2 + \chi k \bigg] & \mapsto & \bigg( \big[\chi_1w_1 + \chi_2w_2\big], \big[c^1\chi_1 + c^2 \chi_2+\chi\big]\bigg)
     \end{array}
   \end{align}
 \end{lem}
 \begin{proof}
   Direct calculation.
 \end{proof}
We give $T_\Lambda$ the induced metric and orientation and choose a Hermitian line bundle $L$ over it with $c_1(L) = h$ (in the standard identification of $H^2\big(T_\Lambda; \Z\big)$ with $\Z$). Furthermore, we equip the bundle with an arbitrary unitary connection $\nabla^L$. 

\begin{defi}
  We define $K:= \pi_{\overline k}^{-1} (L)$ and $\nabla^K := \pi_{\overline k}^{-1} \big(\nabla^L\big)$. Then we have $c_1(K) = \hat k$.
\end{defi}

\subsubsection{Working on $T_\Lambda$}
\label{sec:worktl}

We now look at the corresponding problem on $T_\Lambda$. For each (positive) Chern class $h$, we have an associated bundle $\underline \HQ \otimes L$ over $T_\Lambda$. Then each closed one-form $\alpha_\Lambda$ over $T_\Lambda$ defines a Dirac operator
\begin{align*}
  \dirac^L_{\alpha_\Lambda} : \Gamma\big(\underline \HQ \otimes L\big) \to \big(\underline \HQ \otimes L\big)
\end{align*}
We give $W$ an arbitrary complex structure and scale everything so that we work on $\C/_{\{1, \tau\}}$ with $\text{im}\, \tau >0$. Now we can equip $L$ with a \emph{holomorphic} structure; we choose it so that  $\nabla^L+\ci \alpha_\Lambda$  becomes the Chern connection of the holomorphic bundle.

This specifies a  problem for twisted Dirac operators on a Riemann surface. We use the results of \cite[section 5.2]{tejeroprieto}, where the eigenspaces of $\dirac^L_{\alpha_{\Lambda}}$ are described in terms of holomorphic sections.

The eigenspaces can be made explicit using theta functions. A detailed discussion of all calculations and identifications can be found in \cite[section 2.c]{thesismeier}. The result is the following:

\begin{lem}\label{lem:tlambda}
  We can explicitely construct a basis of orthogonal eigensections $\sigma_m$, $m\in \Z$, for $\dirac^L_{\alpha_\Lambda}$ with respective eigenvalues
  \begin{align*}
    \mu_m := \sgn m \, \sqrt{2\pi h \|k\|\left\lfloor\frac{|m|}{h}\right\rfloor}.
  \end{align*}
The eigenvalues are independent of $\alpha_{\Lambda}$.
\end{lem}

\subsubsection{An eigenbasis for $(\dirac^K_{\alpha})^2$}
\label{sec:specd2}
\begin{rmk}
 By a standard gauging argument, we can reduce the problem of finding spectrum and eigenspaces from closed one-forms to harmonic one-forms. So from now on we assume $\alpha\in H^1(T^3;\R) \cong \R^3$.
\end{rmk}

We now look at the map $s_l \circ \tr$, $l\in \Z$, where $s_l : \R/\Z \to S^1$ is defined to be $t \mapsto \exp(2\pi l t)$ and $\tr$ is the map from (\ref{trivmap}). Its exterior derivative is given by:
\begin{align*}
   d\big(s_l \circ \tr\big) &=  2\pi \ci l \big(s_l \circ \tr\big)\, \big(c^1, \, c^2, \, 1\big).
\intertext{We now want to separate this form into its parallel and orthogonal part with respect to $W$:}
  d\big(s_l \circ \tr\big) &= 2\pi \ci (s_l \circ \tr) \cdot \big(\omega^l_{\shortparallel} + \omega_\perp^l\big),\end{align*} 
In the same way we split $\alpha = \alpha_{\shortparallel} + \alpha_{\perp}$.

We set $\alpha_\Lambda:= \alpha_{\shortparallel} + 2\pi\omega^l_{\shortparallel}$ and use Lemma \ref{lem:tlambda} to determine a basis of sections for $\Gamma(\underline{\HQ} \otimes L)$ which we call $\sigma_m^l$, $m \in \Z$.

The parameter $\omega^l_{\shortparallel}$ becomes necessary for our construction since the bundle $T^3 \to T_\Lambda$ is trivial but its metric differs from the orthogonal product $T_\Lambda \times S^1$. 

\begin{defi}
Define
\begin{align*}
  \hat\sigma_{l,m}(v) &:= (s_l \circ \tr)(v) \cdot \pi_{\overline k}^* \big(\sigma^l_m\big)(v)
\end{align*}
\end{defi}

This can be interpreted as a combination of a basis of the Dirac operator over $S^1$ with bases over $T_\Lambda$. 

\begin{defi}
  Let $\lambda_l:= \big(2\pi l + \langle k, \, \alpha\rangle\big)/\|k\|$, where $\langle\, , \, \rangle$ means the standard scalar product of $\R^3$ (or, interpreted differently, the evaluation of $k\cup \alpha$ at the orientation class). 
\end{defi}

 \begin{thm}[Eigenbasis for $(\dirac^K_{\alpha})^2$]\label{eigend2}
The set $\big\{\hat \sigma_{l,m}\, \big| \, l,m\in \Z \big\}$ forms an orthogonal basis of eigensections for $(\dirac^K_{\alpha})^2$ with the respective eigenvalues $\lambda_l^2 + \mu_m^2$.
\end{thm}
\begin{proof}
  Applying $\dirac^K_\alpha$ twice and using the definition of $\omega^l$, we see that these sections are indeed eigensections for the given eigenvalues. With a standard calculation (see \cite[p.45]{thesismeier}), we conclude that the set $\spann\big\{\hat \sigma_{l,m}\, \big| \, l,m\in \Z \big\}$ is dense in the space of $L^2$-sections. The orthogonality can be deduced from the orthogonality of the $\sigma_m^l$ by using the fact that a change of $\alpha_\perp$ changes the spectrum but fixes $\sigma^l_m$. 
\end{proof}

\subsubsection{An eigenbasis for $\dirac^K_{\alpha}$}
\label{sec:specd}

Theorem \ref{eigend2} gives a quadratic equation for $\dirac^K_{\alpha}$. Furthermore, we know that the Dirac operator on $T_\Lambda$ is graded, so the bases $\sigma^l_m$ split into $\sigma^{l+}_m + \sigma^{l-}_m$. Together this leads us to the following definition:

\begin{defi}
  Let
  \begin{align*}
    \sigma_{l,m}^{\pm} &:= (s_l \circ \tr) \cdot \bigg(\Big(\lambda_l + \mu_m \pm \sqrt{\lambda_l^2 + \mu_m^2}\Big)\, \pi_{\overline k}^*\big( \sigma^{l+}_m\big)\\ &\quad + \Big(-\lambda_l + \mu_m \pm \sqrt{\lambda_l^2 + \mu_m^2}\Big)\, \pi_{\overline k}^*\big( \sigma^{l-}_m\big)\bigg) \\
\sigma_{l,m}^0 &:= \hat \sigma_{l,m}
  \end{align*}
and
\begin{align*}
  \nu_{l,m}^{\pm} &:= \pm \sqrt{\lambda_l^2 + \mu_m^2} \\
  \nu_{l,m}^0 &:= \left\{\begin{array}{ll} \lambda_l & \text{for $0 \leq m \leq h-1$} \\ \mu_m & \text{otherwise}\end{array}\right.
\end{align*}
\end{defi}

From this set of vectors we have to choose a subset of nonzero vectors whose span is dense. 

\begin{thm}\label{basisd}
  We get an orthogonal eigenbasis of $\dirac_\alpha^K$ by
  \begin{align*}
    \bigg\{ \sigma_{l,m}^\pm &\, \bigg|\, (l,m) \in \Z^2 \quad \text{with $\lambda_l \neq 0$ and $m \geq h$}\bigg\}\\ &\cup 
    \bigg\{ \sigma_{l,m}^0 \, \bigg| \, (l,m) \in \Z^2 \quad \text{with $\lambda_l =0$ or $0\leq m \leq h-1$} \bigg\},
  \end{align*}
which will be written as $M_\alpha^{\pm} \cup M_\alpha^0$. The respective eigenvalues are $\nu^{+/0/-}_{l,m}$.
\end{thm}
\begin{proof}
  We check that all these vectors are nonzero and belong to the defined eigenspaces. 

From the construction in \cite{tejeroprieto} we know that $\sigma^l_m = \sigma^{l+}_m + \sigma^{l-}_m$ implies $\sigma^l_{h-m-1} = \sigma^{l+}_m - \sigma^{l-}_m$. 

Therefore, we have the $\dirac_\alpha^K$-invariant subspaces
\begin{align*}
  \spann\big\{\hat \sigma_{l,m},\,\dirac_\alpha^K\, \hat \sigma_{l,m}\big\} = \spann\big\{\hat \sigma_{l,m},\, \hat \sigma_{l,h-m-1}\big\}
\end{align*}
They can be used to prove the orthogonality and density of the constructed sections.
\end{proof}

\subsection{Trivial $\Spinc$ structure}
\label{sec:spectriv}

We look at $\dirac_\alpha$ on $\Gamma(\underline \HQ)=\Gamma(\underline \C^2)$ for the standard connection $\nabla^K$. 

Let \begin{align*}
  \sigma_b(x_1,x_2,x_3):= \exp \big(2\pi \ci (b_1x_1+ b_2x_2 + b_3x_3)\big)\\
\end{align*}
Then we get the basis of sections:
\begin{align*}
  \spann \big\{ \sigma^+_b = (\sigma_b, 0) \, \big|\, b \in \Z^3 \big\} \cup \big\{ \sigma^-_b = (0,\sigma_b) \, \big|\, b \in \Z^3 \big\}
\end{align*}
Define $\beta= \alpha + 2\pi b$.

We use the classical methods of \cite{friedrich} to determine:

\begin{thm}\label{trivialeigend}
  We get an orthogonal eigenbasis for $\dirac_\alpha$ as
  \begin{align*}
    \bigg\{\|\beta\|\sigma_b^+ - \dirac_\alpha \sigma^+_b\, &\bigg| \, b \in \Z^3 \text{ with $\beta_2 \neq 0$ or $\beta_3 \neq 0$}\bigg\} \\
\cup \,\, \bigg\{\|\beta\|\sigma_b^+ + \dirac_\alpha \sigma^+_b\, &\bigg| \, b \in \Z^3 \text{ with $\beta_2 \neq 0$ or $\beta_3 \neq 0$}\bigg\} \\
\cup \,\, \bigg\{ \sigma^\pm_b\, &\bigg|\, \beta_2 = \beta_3 = 0 \bigg\}.
  \end{align*}
  Furthermore, we have for $\beta_2 \neq 0$ or $\beta_3 \neq 0$:
  \begin{align*}
    \spann\big\{\sigma_b^+,\, \sigma_b^-\big\} = \spann\big\{ \|\beta\| \sigma^+_b - \dirac_\alpha \sigma^+_b, \,  \|\beta\| \sigma^+_b + \dirac_\alpha \sigma^+_b \big\}.
  \end{align*}
The spectrum consists of all numbers $\pm \|\beta(b, \alpha)\|$ for $b\in \Z^3$. 
\end{thm}

\begin{rmk}
  In the case $\hat k \neq 0$ the spectrum is determined by $\alpha_\perp$ while the eigenbasis is determined by $\alpha_{\shortparallel}$. Here every change of $\alpha$ has influence on both eigenbasis and spectrum. 
\end{rmk}

\section{Spectral sections}
\label{sec:specsec}

We look at families of Dirac operators over a compact base space $B$. \cite{melrosepiazza} defined the concept of a \emph{spectral section} for a constant $R>0$. The most interesting spectral sections are those for small $R$; they should be classified in the sense of the following definition.

\begin{defi}\label{infspecsec}
  Let $R_{\text{inf}}$ be defined as the infimum of the set 
\begin{align*} \{R>0\, |\, \text{for $R$ exists at least one spectral section}\}. \end{align*}
Furthermore, choose a (small) positive number $\eps_P$. Then a \emph{system of infinitesimal spectral sections}\index{system of infinitesimal spectral sections} is a map
  \begin{align*}
    \begin{array}{rcl}
      \big]R_{\text{inf}}, R_{\text{inf}}+\eps_P\big] \times I & \rightarrow & \big\{\text{spectral sections for a fixed operator $D$}\big\} \\[1ex]
      (R,i) & \mapsto & P^i_R,
    \end{array}
  \end{align*}
where
\begin{enumerate}
\item $I$ is an arbitrary index set,
\item $P^i_R$ is a spectral section for the constant map $R$,
\item every $\big(P^i_R\big)_\alpha $, $\alpha \in B$, depends continuously on $R$ (where we consider $\big(P^i_R\big)_\alpha $ as operator between L${}^2$ spaces), and
\item $\cup_{i\in I} \{P^i_R\}$ is a representation system for all spectral sections for $R$, i.e. for all possible spectral sections $P_R$ there is a $P^i_R$ with $i\in I$, so that $\im P_R - \im P^i_R$ is zero in $K$-theory.
\end{enumerate}

A \emph{minimal system of infinitesimal spectral sections} is one in which $I$ is chosen minimal (under the inclusion relation).
\end{defi}

\subsection{Definition of the family}
\label{sec:deffam}

Let $\ell\subset H^1\big(T^3;\Z\big)$ be a lattice (of non-maximal dimension) and let $B:= (\ell \otimes \R)/\ell$. 

We need the following ingredients for our definition:
\begin{itemize}
\item  $\ker(d)_{l\otimes \R}$:  The subset of $\ker(d)$ representing elements in $\ell \otimes \R$.
\item  ${\cal G}_\ell$: The subgroup of the gauge group $\text{Map}\big(T^3, S^1\big)$ determined by $\ell$.
\item The projection  $\pr_{T^3}: T^3 \times \big( \nabla^K + \ci \ker(d)_{l\otimes \R}\big) \rightarrow T^3$ together with the induced vector bundle $\pr_{T^3}^*\big(\underline{\HQ} \otimes K\big)$.
\end{itemize}

If $v$ is an element of the fibre of $\pr_{T^3}^*\big(\underline{\HQ} \otimes K\big)$ over
\begin{align*}
   (y, \nabla^K + \ci \alpha^c) &\in  T^3\times \big(\nabla^K + \ci\ker(d)_{\ell \otimes \R}\big), 
\end{align*}
we can define the following action of $\mathcal{G}_\ell$:
 \begin{align}
   \mathcal{G}_\ell \times \pr_{T^3}^*\big(\underline{\HQ} \otimes K\big) &\to \pr_{T^3}^*\big(\underline{\HQ} \otimes K\big)\nonumber \\
   \label{bllaction}\bigg(u\, , \, \big(v,y, \nabla^K + \ci \alpha\big)\bigg) & \mapsto \big(u(y)\cdot v, y, \nabla^K + \ci \alpha + udu^{-1}\big),
 \end{align}
The quotient is a bundle over $T^3 \times B$. The connection from the parameter space determines a family of Dirac operators called $\dirac$.

Depending on $\hat k$ and $\ell$ we want to know:
\begin{enumerate}
\item Do spectral sections exist?
\item If they exist: What do they look like?
\end{enumerate}

\subsection{Existence of spectral sections}
\label{sec:exisspec}
Following \cite{melrosepiazza} we know that spectral sections for $\dirac$ exist if and only if the index of $\dirac$ in $K^1(B)$ vanishes. Let ${\cal I}$ be the following composition of isomorphisms (remember that $B$ is a torus of maximal dimension 2):
\begin{align*}
  K^1(B) \stackrel{\text{Chern}}{\longrightarrow} H^1(B;\Z) \longrightarrow \big(H_1(B;\Z)\big)^* \longrightarrow \ell^*
\end{align*}

\begin{lem}\label{lem:flow}
  Let $a\in H^1\big(T^3; \Z\big)$ and let $f: (\R \cdot a) /a \to B$ be the map induced by the inclusion. In this way we get a pullback family $\dirac^a$ over $(\R \cdot a)/a$. Then the spectral flow of $\dirac^a$ in positive direction is given by \begin{align*}\langle \hat k,\, a \rangle = \Big\langle \hat k\, \cup\, a, \, \big[T^3\big]\Big\rangle\end{align*}
    

\end{lem}
\begin{proof}
  We use our explicit knowledge of the spectrum.

First we assume $\hat k \neq 0$:  From all eigenvalues $\nu^{+/0/-}_{l,m}$ only those of the form $\nu^0_{l,m}$ for $0\leq m \leq h-1$ have a chance to cross zero. From the definition we know that  $\nu^0_{l,m}=\lambda_l=\big(2\pi l + \langle k, \, \alpha\rangle\big)/\|k\|$ for which we can count the crossings while running around the circle.

For $\hat k = 0$ the spectrum is always symmetric with respect to zero. We see that every spectral flow has to vanish.
\end{proof}
With this Lemma we get a direct access to the following statement: 
\begin{thm}
  The isomorphism ${\cal I}$ maps the index of $\dirac$ to the map $x\mapsto \Big\langle \hat k\, \cup\, x, \, \big[T^3\big]\Big\rangle$ in $\ell^*$.
\end{thm}
\begin{proof}
  Take a fundamental basis $a_1, a_2$ of the torus $B$; then an element in $K^1(B)$ is determined by its images in $K^1\big((\R \cdot a_i)/a_i\big)$, which we calculate with the formula from the preceding lemma. Since the maps are linear, it is enough to check the theorem for $a_1, a_2$ which is an easy exercise.
\end{proof}
\begin{cor}
  Spectral sections for $\dirac$ exist if and only if $k \cup \ell = 0$.
\end{cor}

\subsection{Construction of spectral sections for $\hat k \neq 0$}
\label{sec:specknon}

\begin{thm}
  If spectral sections exist, the spectrum is constant.
\end{thm}
\begin{proof}
From $k \cup \ell =0$ we know that for every $\alpha \in (\ell \otimes \R)$ we have $\alpha_\perp = 0$. From section \ref{sec:specd} we know that this implies a constant spectrum.
\end{proof}

Therefore, we have $R_{\text{inf}}=0$. For $\eps_P$ smaller than the smallest eigenvalue of $\dirac$, the spectral sections are fixed everywhere except for the $h$-dimensional kernel of $\dirac$.

Let $ I := \big\{F\, \big|\, F\, \text{subbundle of $B\times \C^h$}\big\}/_{\cong}\,\cong \Z^{h-1} \cup \{0\} \cup \{\C^k\}$  and define $P_F|_{\ker \dirac}$ for $R<\eps_P$ as the orthogonal projection onto $F$. This defines a system of infinitesimal spectral sections which is obviously also minimal.

\subsection{Construction of spectral sections for $\hat k = 0$}
\label{sec:speckzer}

We split $\Gamma_{L^2}(\HQ)$ into the 2-dimensional $\dirac_\alpha$-invariant subspaces $\Sigma_b = \spann\{\sigma_b^+,\, \sigma_b^-\}$. On each of them, we have the two eigenvalues $\pm \|\beta\|= \pm \|\alpha+2\pi b\|$. For small $R$ we know that for each $\alpha$ there is at most one $b$ with $\|\beta\|\leq R$. So for any spectral section $P$  for $\dirac$ with small $R$ we know that it fixes all $\Sigma_b$. Since $P_\alpha|\Sigma_b: \Sigma_b \to \Sigma_b$ is a one-dimensional orthogonal projection for $\|\beta\|>R$, it has to be a one-dimensional orthogonal projection for all $\beta$ (and, therefore, for all $\alpha$, since $\alpha$ and $\beta$ are in bijective correspondence).

We now assume that $\ell$ is a plane since $\dim \ell \leq 1$ does not lead to interesting conclusions. In addition to the assumptions about $R$ above we  assume that $\eps_P$ is smaller than the minimal distance between $\ell \otimes \R$ and any point $b \in \Z^3 \backslash \ell$. This implies that for such $b$ there are no eigenvalues with $\|\beta\|<R$ on $\Sigma_b$.

The space of one-dimensional orthogonal projections on $\C^2$ equals $\C\Peins \cong S^2$. Fix an element $b \in \ell_\Z = (\ell \otimes \R) \cap \Z^3$ and look at the corresponding map $P_\beta|_{\Sigma_b}: \ell \otimes \R \to \C\Peins$ (written as function of $\beta$). For $\|\beta\|\geq R$ every ray coming from zero will be mapped to one point, producing a circle in $\C\Peins$ (this follows from the construction of the eigenbasis).

For $\|\beta\|<R$ we have to continue this map in some way; topologically, the problem is as follows: We have to construct a map from the 2-disc to the 2-sphere which maps the boundary pointwise to the equator. Up to homotopy, there are $\pi_2(S^2) \cong \Z$ many choices for that.

\subsubsection{A system of infinitesimal spectral sections}
\label{infspec}

The preceding discussion leads to the following:

Since we had imposed no lower bounds for $R$, we have $R_{\text{inf}}=0$. Let $\eps_P$ be so small that if fulfills all conditions mentioned above.

We take ${I} = \Big\{ g:\ell_\Z/\ell \to \pi_2\big(\C\Peins\big)\Big\}$ and define for each $R<\eps_P$ spectral projections $P^g$. For $b\not\in \ell_\Z$ these maps are already defined on $\Sigma_b$. For $b\in \ell_Z$, we define $P^g_\alpha$ on $\Sigma_b$ to be a continuation specified by $g(b)\in \pi_2\big(\C\Peins\big)$ as discussed in the preceding subsection (These continuations can be chosen to depend continuously on the parameters).

Conditions 1 and 2 (from the definition of infinitesimal spectral sections) are clear, 3 can be checked directly (if we specify the continuations explicitly), and 4 follows from the discussion above.

In general this system is not minimal. We can choose a minimal system $J$ by fixing an element $g_0 \in I$ and a point $l_0 \in \ell_{\Z}/\ell$ and defining
\begin{align*}
  J &= \big\{ g \in I\, \big|\, g(l) = g_0(l) \quad \text{for $l\neq l_0$}\big\}.
\end{align*}
This is true because $J$ represents all element of the form $(0,z)$ from $K(B) \cong H^0(B;\Z) \oplus H^2(B;\Z) \cong \Z\oplus \Z$.



\section{Acknowledgements}
\label{sec:ack}

This article grew out of my dissertation \cite{thesismeier}. I would like to thank my supervisor Prof. Stefan Bauer for his support. Furthermore, I thank Johannes Ebert for helpful suggestions.





\bibliographystyle{plain}
\bibliography{hitchin}







\end{document}